\documentclass[fullpaper,9pt]{IEEEtran}

\usepackage{amsmath,amssymb,enumerate,subfigure,multirow,hhline,color}
\usepackage{xcolor}
\usepackage{hyperref}
\usepackage{graphicx}
\usepackage{graphics}
\usepackage[sort]{cite}
\usepackage{amsthm}
\usepackage{algorithm}
\usepackage{algpseudocode}

\usepackage{tcolorbox}
\usepackage{epsfig,psfrag}
\usepackage{tabularx}
\usepackage{tabulary}

\usepackage[sort]{cite}

\usepackage{hyperref}
\hypersetup{breaklinks=true,colorlinks=true,linkcolor=blue,citecolor=blue}

\usepackage[noabbrev,capitalise,nameinlink]{cleveref}

\DeclareMathOperator{\argmin}{arg\,min}

\newtheorem{theorem}{Theorem}

\newtheorem{assumption}{Assumption}

\newtheorem{problem}{Problem}
\newtheorem{lemma}{Lemma}
\newtheorem{definition}{Definition}
\newtheorem{proposition}{Proposition}

\allowdisplaybreaks

\title{Data-Driven Control Design with LMIs and Dynamic Programming}

\author{Donghwan Lee and Do Wan Kim
\thanks{D. Lee is with the Department of Electrical Engineering,
KAIST, Daejeon, 34141, South Korea {\tt\small
donghwan@kaist.ac.kr}.}
\thanks{D. Kim is with the Department of Electrical Engineering, Hanbat National University, Daejeon 34158, South Korea {\tt\small
dowankim@hanbat.ac.kr}.}
}

\begin{document}

\maketitle

\begin{abstract}
The goal of this paper is to develop data-driven control design and evaluation strategies based on linear matrix inequalities (LMIs) and dynamic programming. We consider deterministic discrete-time LTI systems, where the system model is unknown. We propose efficient data collection schemes from the state-input trajectories together with data-driven LMIs to design state-feedback controllers for stabilization and linear quadratic regulation (LQR) problem. In addition, we investigate theoretically guaranteed exploration schemes to acquire valid data from the trajectories under different scenarios. In particular, we prove that as more and more data is accumulated, the collected data becomes valid for the proposed algorithms with higher probability. Finally, data-driven dynamic programming algorithms with convergence guarantees are then discussed.
\end{abstract}

\section{Introduction}
Recently, reinforcement learning (RL)~\cite{sutton1998reinforcement} and data-driven control design have captured significant attentions due to its successful demonstrations that outperform humans in several challenging tasks~\cite{mnih2015human,lillicrap2015continuous}. The goal of this paper is to develop efficient data-driven control design methods for deterministic discrete-time linear-time invariant (LTI) systems. Two different lines of approaches are addressed: linear matrix inequalities (LMIs)~\cite{Boyd1994} and dynamic programming~\cite{bertsekas2005dynamic,bertsekas1996neuro}. In particular, we develop simple and efficient data-driven LMIs for stabilization and LQR problems along with data collection algorithms tailored to the proposed LMIs, which allow us to design controllers without the knowledge of the model. We also prove rigorously the optimality of the LMI solutions. Moreover, new data collection algorithms are developed, and we prove that these algorithms guarantee the validity of the data in the probabilistic sense, where the validity implies that it includes sufficient information for the proposed algorithms to successfully solve the given problems. Finally, additional data-driven dynamic programming algorithms are proposed based on the data collection algorithms with their convergence proofs. All these algorithms are sample efficient in the sense that once valid data is collected, then no more data is required to solve the problems completely.

{\bf Related works}: The previous works can be roughly categorized into two parts: RL (or data-driven dynamic programming) and data-based LMIs. As for RL, the early work~\cite{bradtke1994adaptive} proposed a Q-learning algorithm~\cite{watkins1992q} for discrete-time LIT systems, where the approximate Bellman equation is solved using the least-square method and trajectories. More comprehensive least-square reinforcement learning approaches were reported in~\cite{lewis2009reinforcement}. A model-based RL has been studied in~\cite{dean2020sample} for discrete-time LTI systems with sample complexity analysis. A policy gradient algorithm for LTI systems and its global convergence were provided in~\cite{fazel2018global}. An efficient online RL with guaranteed finite-time regret bounds has been proposed in~\cite{cohen2018online} based on a novel semidefinite programming relaxation. The paper~\cite{tu2019gap} proposed several model-based and model-free RLs. \cite{lee2018primal} proposed a policy iteration reinforcement learning based on the Lagrangian duality perspectives of the Bellman equation.

As for the data-based LMIs, several advances have been made recently in deriving numerically tractable data-based LMIs that enable direct data-driven control designs. Data-dependent LMIs were developed in~\cite{dai2018moments} for stabilization of switched systems. The paper~\cite{de2019formulas} introduced a data-dependent controller parameterization, and proposed data-based LMIs for stabilization and optimal control problems. The concept of informative data was introduced in~\cite{van2020data}, from which necessary and sufficient data-based conditions have been developed for various control problems.
The paper~\cite{berberich2020robust} proposed LMI conditions for control with guaranteed stability and performance by introducing a notion of noise bounds. Recently,~\cite{van2021matrix} introduced data-driven LMI conditions for stabilization problems based on a matrix version of the classical Finsler’s lemma~\cite{skelton1997unified}.

{\bf Contribution}: Compared to the previous works, the proposed data-driven LMIs provide more intuitive conditions with more memory efficient data structures and computational efficiency in terms of the size of LMIs. We additionally provide data-based LMIs for policy evaluations.
Moreover, new data generation schemes are developed with different scenarios. We prove that the new data collection approaches is guaranteed to be valid with probability one as more and more trajectories are accumulated. Lastly, data-driven dynamic programming schemes are briefly discussed, whose learning process is off-policy. The algorithms are sample efficient in the sense that once the data is collected, then no more samples are required.

{\bf Notation}: The adopted notation is as follows: ${\mathbb R}$: set of real numbers; ${\mathbb R}^n $: $n$-dimensional Euclidean
space; ${\mathbb R}^{n \times m}$: set of all $n \times m$ real
matrices; $A^T$: transpose of matrix $A$; $A^{-T}$: transpose of matrix $A^{-1}$; $A \succ 0$ ($A \prec
0$, $A\succeq 0$, and $A\preceq 0$, respectively): symmetric
positive definite (negative definite, positive semi-definite, and
negative semi-definite, respectively) matrix $A$; $I$: identity matrix with appropriate dimensions; ${\mathbb S} ^n $: symmetric $n \times
n$ matrices; ${\mathbb S}_+^n $: cone of symmetric $n \times n$
positive semi-definite matrices; ${\mathbb S}_{++}^n$:
symmetric $n\times n$ positive definite matrices; ${\bf Tr}(A)$:
trace of matrix $A$; $\rho(\cdot)$: spectral radius; ${\rm{diag}}(A_1 , \ldots ,A_n )$: block diagonal matrix with diagonal elements $A_1 , \ldots ,A_n$.
.

\section{Problem formulations and preliminaries}

Consider the LTI system
\begin{align}
&x(k + 1) = Ax(k) + Bu(k),\quad x(0) = z \in {\mathbb
R}^n,\label{eq:LTI-system}
\end{align}
where $k \in {\mathbb N}$, $x(k) \in {\mathbb R}^n$ is the state
vector, $u(k) \in {\mathbb R}^m$ is the input vector, and $z \in
{\mathbb R}^n$ is the initial state.

Assuming the control $u(k)$ is given by a state-feedback control
policy $u(k)=Fx(k)$, we denote by $x(k;F,z)$ the solution of
\eqref{eq:LTI-system} starting from $x(0)=z$. Under the
state-feedback control policy, the cost function for the classical
LQR problem is denoted by
\begin{align}
&J(F,z):= \sum_{k = 0}^\infty{ \begin{bmatrix}
   x(k;F,z)\\
   Fx(k;F,z)\\
\end{bmatrix}^T \Lambda \begin{bmatrix}
   x(k;F,z)\\
   Fx(k;F,z)\\
\end{bmatrix}},\label{eq:cost-function}
\end{align}
where $\Lambda := \begin{bmatrix}
   Q & 0\\
   0 & R\\
\end{bmatrix}\succeq 0$ is the weight matrix.

By introducing the augmented state vector $v(k):=
\begin{bmatrix}
  x(k)\\
  u(k)\\
\end{bmatrix}$, we will consider the augmented system
\begin{align}
&v(k+1)=A_F v(k),\quad v(0) = v_0  \in {\mathbb R}^{n +
m},\label{eq:augmented-system}
\end{align}
where $A_F := \begin{bmatrix}
  A &B \\
  FA &FB \\
\end{bmatrix} \in {\mathbb R}^{(n+m)\times(n+m)}$, which plays an important role throughout the paper. A useful
property of $A_F$ is that its spectral radius $\rho (A_F)$ is
identical to that of $A+BF$.
\begin{lemma}[\cite{lee2018primal}]\label{lemma:spetral-radius-lemma}
$\rho(A+BF)= \rho (A_F)$ holds.
\end{lemma}

Define ${\cal F}$ as the set of all stabilizing state-feedback
gains of system $(A,B)$.
\begin{definition}[Stabilizing set]
The set of all stabilizing state-feedback gains of system $(A,B)$ is denoted by
\begin{align*}
{\cal F}:= \{F \in {\mathbb R}^{m \times n} :\rho(A+BF)<1\}
\end{align*}
\end{definition}

Note that ${\cal F}$ is an open set, and not necessarily convex~\cite[Lemma~2]{geromel1998static}. However, finding a
state feedback gain $F\in {\cal F}$ can be reduced to a simple
convex problem. In this paper, we study both the LQR
problem and stabilization problem.
\begin{problem}[Stabilization problem]\label{problem:stabilization}
Find a stabilizing feedback gain $F \in {\cal F}$.
\end{problem}
\begin{problem}[LQR problem]\label{problem:infinite-horizon-LQR}
Solve $F^*=\argmin_{F\in {\mathbb R}^{m \times n}} J(F,z)$ if the optimal value of $\inf_{F\in {\mathbb R}^{m \times n}} J(F,z)$ exists and is attained.
\end{problem}

From the standard LQR theory, although $J^* (F,z)$ has different
values for different $z \in {\mathbb R}^n$, the minimizer
$F^*=\argmin_{F \in {\mathbb R}^{m \times n}} J(F,\,z)$ is not dependent on $z$.
Therefore, it follows that $\argmin_{F \in {\mathbb R}^{m \times n}} J(F,\,z) = \argmin_{F \in {\mathbb R}^{m \times n}} \sum\limits_{i = 1}^r {J(F,\,z_i )}$ for any $z,z_i \in {\mathbb R}^n,\,i \in \{ 1,\,2, \ldots ,\,r\}$. For technical reasons that will become clear later, we solve
\begin{align*}
F^*:=\argmin _{F \in {\mathbb R}^{m \times n}} \sum_{i=1}^n {J(F,e_i )}
\end{align*}
instead of $\argmin_{F \in {\mathbb R}^{m \times n}} J(F,z)$, where $e_i \in {\mathbb R}^n$ is the $i$th standard basis vector.
Therefore, it will be useful to define a standard measure of the cost. In this paper, we will use the following cost index:
\[
J(F): = \sum_{i = 1}^n {J(F,e_i )}
\]

For a given $z\in {\mathbb R}^n$, if the optimal value of $\inf_{F \in {\mathbb R}^{m \times n}} J(F,z)$ exists and is attained,
then the optimal cost is denoted by $J^*(z)=J(F^*)$. Assumptions that will be used throughout the paper are summarized
below.
\begin{assumption}\label{assumption:basic-assumption}
Throughout the paper, we assume that
\begin{itemize}
\item $Q \succeq 0,R \succ 0$;

\item $(A,B)$ is stabilizable, and $Q$ can be written as $Q=C^T
C$, where $(A,C)$ is detectable.
\end{itemize}
\end{assumption}

Under~\cref{assumption:basic-assumption}, the optimal
value of $\inf_{F \in {\mathbb R}^{m \times n}} J(F)$ exists, is attained, and
$J^*(z)$ is a quadratic function, i.e., $J^*(z) = z^T X^* z$,
where $X^*$ is the unique solution of the algebraic Riccati
equation (ARE)~\cite[Proposition~4.4.1]{bertsekas2005dynamic} for $X$:
\begin{align*}
&X = A^T XA - A^T XB(R + B^T XB)^{-1} B^T XA + Q,\quad X\succeq 0.
\end{align*}
In this case, $J^*(z)$ as a function of $z\in {\mathbb R}^n$ is
called the optimal value function. The reader can refer to~\cite{bertsekas2005dynamic}
and~\cite{kwakernaak1972linear} for more details of the classical
LQR results. The corresponding optimal control policy is $u^* (z)
= F^* z$, where
\begin{align}
&F^*:=-(R + B^T X^* B)^{-1} B^T X^*A \in {\cal F}\label{eq:F*}
\end{align}
is the unique optimal gain. Alternatively, the
$Q$-function~\cite{bertsekas2005dynamic} is defined as
\begin{align}
&Q^*(z,u):=z^T Qz + u^T Ru + J^*(Az+Bu)= \begin{bmatrix}
   z  \\
   u  \\
\end{bmatrix}^T P^* \begin{bmatrix}
   z  \\
   u  \\
\end{bmatrix},\label{eq:Q-factor}
\end{align}
where
\begin{align}
&P^*:=\begin{bmatrix}
   Q + A^T X^* A & A^T X^* B \\
   B^T X^*A & R + B^T X^*B  \\
\end{bmatrix}.\label{eq:P*}
\end{align}
The optimal policy in terms of the Q-function is then given by
\begin{align*}
&u^*(z)=F^*z=\argmin _{u \in {\mathbb R}^m}Q^*(z,\,u).
\end{align*}

Before closing this section, some useful lemmas are summarized.
\begin{lemma}[\cite{skelton1997unified}]\label{lemma:Finsler} For a vector $x \in \mathbb{R}^n $ and two matrices $Q = Q^T\in\mathbb{R}^{n \times
n}$ and $R\in\mathbb{R}^{m \times n}$ such that ${\rm
rank}\left(R\right)<n$, the following statements are equivalent:
\begin{enumerate}[1)]
\item $x^T Qx < 0,\quad \forall x \in \left\{{x\in\mathbb{R}^n
\left| {x\ne 0,\,Rx = 0}\right.}\right\},$

\item $\exists M \in\mathbb{R}^{n \times m}$ such that $Q + MR +
R^T M^T \prec 0$.
\end{enumerate}
\end{lemma}
\begin{lemma}\label{lemma:bounding-lemma}
Given matrices $U,V$ of appropriate dimensions, the following holds for any $\varepsilon > 0$:
\[
 - \varepsilon^{- 1} U^T U - \varepsilon V^T V \preceq U^T V + V^T U \preceq \varepsilon ^{ - 1} U^T U + \varepsilon V^T V.
\]
\end{lemma}
\begin{proof}
The first inequality comes from $(\varepsilon ^{ - 1/2} U + \varepsilon ^{1/2} V)^T (\varepsilon ^{ - 1/2} U + \varepsilon ^{1/2} V) = \varepsilon ^{ - 1} U^T U + U^T V + V^T U + \varepsilon V^T V \succeq 0$ and the reversed inequality is obtained from $(\varepsilon ^{ - 1/2} U - \varepsilon ^{1/2} V)^T (\varepsilon ^{ - 1/2} U - \varepsilon ^{1/2} V) = \varepsilon ^{ - 1} U^T U - U^T V - V^T U + \varepsilon V^T V\succeq 0$. This completes the proof.
\end{proof}

\section{Data collection}
In this section, we introduce two data acquisition schemes, which will be used for the main algorithms. In particular,~\cref{algo:data-collection2} will be called an on-policy data collection algorithm with exploring starts, $(S(F),H(F))={\tt On-Collect}(F)$, where on-policy means that the generated data depends on a particular state-feedback gain $F$. The data generated by~\cref{algo:data-collection2} can be useful when we want to evaluate the specific state-feedback gain $F$, i.e., its stabilizability or the LQR performance. The exploring starts~\cite{sutton1998reinforcement} imply that for sufficient exploration of the state-space,~\cref{algo:data-collection2} needs trajectories starting from different initial states $x(0) = e_i$, $i=1,2,\ldots,n$, where $(e_1,e_2,\ldots, e_n)$ is the standard basis that spans the state-space, ${\mathbb R}^n$.  The following lemma offers a useful property of the data matrices.
\begin{lemma}[Data matrix transformation]\label{lemma:S-H2}
$S(F)A_F^T  = H(F)$ holds.
\end{lemma}
\begin{proof}
We have
\begin{align*}
S(F)A_F^T  =& \frac{1}{{nN}}\sum\limits_{i = 1}^n {\sum\limits_{k = 0}^{N - 1} {\left[ {\begin{array}{*{20}c}
   {x(k;F,e_i )}  \\
   {u(k)}  \\
\end{array}} \right]\left[ {\begin{array}{*{20}c}
   {x(k;F,e_i )}  \\
   {u(k)}  \\
\end{array}} \right]^T } } A_F^T\\
=& \frac{1}{{nN}}\sum\limits_{i = 1}^n {\sum\limits_{k = 0}^{N - 1} {\left[ {\begin{array}{*{20}c}
   {x(k;F,e_i )}  \\
   {u(k)}  \\
\end{array}} \right]\left[ {\begin{array}{*{20}c}
   {x(k + 1;F,e_i )}  \\
   {u(k + 1)}  \\
\end{array}} \right]^T } }\\
=& H(F)
\end{align*}
\end{proof}

\begin{algorithm}[h]
\caption{On-policy data collection $(S(F),H(F))={\tt On-Collect}(F)$ with exploring starts}
\begin{algorithmic}[1]
\State Initialize $S_0  = 0,H_0  = 0$.

\For{$i \in \{1,2,\ldots,n\}$}
\State Initialize $x(0)=e_i$.
\For{$k \in \{0,1,\ldots,N-1\}$}
\State Apply control input $u(k)=Fx(k)$
\State Observe $x(k+1)$

\State Update
\[
S_{k + 1}  \leftarrow \frac{k}{{k + 1}}S_k  + \frac{1}{{k + 1}}\left[ {\begin{array}{*{20}c}
   {x(k)}  \\
   {u(k)}  \\
\end{array}} \right]\left[ {\begin{array}{*{20}c}
   {x(k)}  \\
   {u(k)}  \\
\end{array}} \right]^T
\]
\[
H_{k + 1}  \leftarrow \frac{k}{{k + 1}}H_k  + \frac{1}{{k + 1}}\left[ {\begin{array}{*{20}c}
   {x(k)}  \\
   {u(k)}  \\
\end{array}} \right]\left[ {\begin{array}{*{20}c}
   {x(k + 1)}  \\
   {u(k + 1)}  \\
\end{array}} \right]^T
\]

\EndFor
\EndFor

\State Return $(S(F),H(F)) = (S_N, H_N)$

\end{algorithmic}\label{algo:data-collection2}
\end{algorithm}
\begin{algorithm}[h]
\caption{Off-policy data collection $(S,H)={\tt Off-Collect}(z)$ with exploration}
\begin{algorithmic}[1]
\State Initialize $S_0  = 0,H_0  = 0$.
\State Initialize $x(0)=z$.
\State Initialize $\varepsilon >0$.

\For{$k \in \{0,1,\ldots\}$}
\State Apply control input $u(k)$ with some excitation input
\State Observe $x(k+1)$

\State Update
\[
S_{k + 1}  \leftarrow \frac{k}{{k + 1}}S_k  + \frac{1}{{k + 1}}\left[ {\begin{array}{*{20}c}
   {x(k)}  \\
   {u(k)}  \\
\end{array}} \right]\left[ {\begin{array}{*{20}c}
   {x(k)}  \\
   {u(k)}  \\
\end{array}} \right]^T
\]
\[
H_{k + 1}  \leftarrow \frac{k}{{k + 1}}H_k  + \frac{1}{{k + 1}}\left[ {\begin{array}{*{20}c}
   {x(k)}  \\
   {u(k)}  \\
\end{array}} \right]x(k + 1)^T
\]

\If{$\lambda _{\min } (S_{k + 1} ) > \varepsilon$
}
\State Stop and return $(S,H) = (S_{k+1},H_{k+1})$
\EndIf

\EndFor

\end{algorithmic}\label{algo:data-collection1}
\end{algorithm}

Another method,~\cref{algo:data-collection1}, is an off-policy data collection algorithm with exploration. Here, the off-policy indicates that the data generated by~\cref{algo:data-collection1} does not depend on a specific state-feedback gain, and it is particularly useful for design algorithms. Roughly speaking, the exploration means that it uses some exploration signals in control inputs to sufficiently explore the state-space so as to collect sufficient information on the model. Note that the data matrices, $(S,H)$, in~\cref{algo:data-collection1} can be expressed as
\[
S: = \frac{1}{N}\sum\limits_{k = 0}^{N - 1} {\left[ {\begin{array}{*{20}c}
   {x(k)}  \\
   {u(k)}  \\
\end{array}} \right]\left[ {\begin{array}{*{20}c}
   {x(k)}  \\
   {u(k)}  \\
\end{array}} \right]^T }
\]
\[
H: = \frac{1}{N}\sum\limits_{k = 0}^{N - 1} {\left[ {\begin{array}{*{20}c}
   {x(k)}  \\
   {u(k)}  \\
\end{array}} \right]x(k + 1)^T }
\]

Throughout the paper, we call the data generated by the data collection algorithms is valid if the $S$-matrix ($S(F)$ or $S$) is strictly positive definite. For completeness, the definition is formally stated below.
\begin{definition}[Data validity]
The data $(S,H)$ and $(S(F),H(F))$ generated by~\cref{algo:data-collection2} and~\cref{algo:data-collection1}, respectively, is said to be valid if $S > 0$ and $S(F) > 0$, respectively.
\end{definition}

The validity of the data ensures that all the proposed methods perform well, and completely solve the desired problems.
Due to the exploring starts in~\cref{algo:data-collection2}, we can prove that the data from~\cref{algo:data-collection2} is always valid for any $N >0$.
\begin{lemma}[Data validity of~\cref{algo:data-collection2}]
With a positive integer $N >0$, $S(F) \succ 0$ holds.
\end{lemma}
\begin{proof}
We have
\begin{align*}
S(F) =& \frac{1}{{nN}}\sum\limits_{i = 1}^n {\sum\limits_{k = 0}^{N - 1} {\left[ {\begin{array}{*{20}c}
   {x(k;F,e_i )}  \\
   {u(k)}  \\
\end{array}} \right]\left[ {\begin{array}{*{20}c}
   {x(k;F,e_i )}  \\
   {u(k)}  \\
\end{array}} \right]^T } }\\
=&\frac{1}{{nN}}\sum\limits_{i = 1}^n {\sum\limits_{k = 0}^{N - 1} {(A_F )^k e_i e_i^T (A_F^T )^k } }\\
=&\frac{1}{{nN}}\sum\limits_{k = 0}^{N - 1} {(A_F )^k (A_F^T )^k } \\
\succeq& I
\end{align*}
which completes the proof.
\end{proof}

On the other hand,~\cref{algo:data-collection1} cannot theoretically guarantee the validity. Therefore, we adopt the so-called persistent excitation assumption for~\cref{algo:data-collection1}, given below.
\begin{assumption}[Persistent excitation]\label{assumption:persistent-excitation}
There exists a positive integer $N>0$ such that $S\succ 0$ from~\cref{algo:data-collection1}.
\end{assumption}

We notice that it is typical to apply~\cref{assumption:persistent-excitation} in adaptive control and reinforcement learning community~\cite{bradtke1994adaptive,lewis2009reinforcement,aastrom2013adaptive}.
Moreover, in the last section, more sophisticated data collection algorithms will be developed, which theoretically guarantee the data validity with different scenarios. Finally, the following lemma will be useful throughout the paper.
\begin{lemma}[Data matrix transformation]\label{lemma:S-H} The following identity holds:
\[
S\left[ {\begin{array}{*{20}c}
   {A^T }  \\
   {B^T }  \\
\end{array}} \right] = H
\]
\end{lemma}
\begin{proof}
\begin{align*}
S\left[ {\begin{array}{*{20}c}
   {A^T }  \\
   {B^T }  \\
\end{array}} \right] =& \frac{1}{N}\sum\limits_{k = 0}^{N - 1} {\left[ {\begin{array}{*{20}c}
   {x(k)}  \\
   {u(k)}  \\
\end{array}} \right](Ax(k) + Bu(k))^T }\\
=& \frac{1}{N}\sum\limits_{k = 0}^{N - 1} {\left[ {\begin{array}{*{20}c}
   {x(k)}  \\
   {u(k)}  \\
\end{array}} \right]x(k + 1)^T }\\
=& H
\end{align*}
\end{proof}

\section{Data-driven LMIs for stabilization}
In this section, the main focus is on data-driven LMIs for stabilization, where the model $(A,B)$ is unknown. The main breakthrough in this approach lies in augmenting the state and input into a single augmented state as in~\eqref{eq:augmented-system}. Then, the model data $(A,B)$ can be eliminated using the data matrices $(S,H)$ or $(S(F),H(F))$ together with~\cref{lemma:S-H2} and~\cref{lemma:S-H}. We first consider a policy evaluation problem. In this setting, given a potentially unknown state-feedback gain $F$, we only have an access to the state-input trajectories. Under this situation, the problem is to determine whether or not the unknown feedback gain $F$ stabilizes the system.
\begin{proposition}[Stability evaluation]\label{proposition:stab-eval}
The system~\eqref{eq:LTI-system} is stabilizable under $u(k) = Fx(k)$ if and only if there exist $P \in {\mathbb S}^{(n+m)}$ such that the following LMI holds:
\[
H(F)^T PH(F) \prec S(F)PS(F),\quad P \succ 0
\]
\end{proposition}
\begin{proof}
From the Lyapunov theory, $A_F^T$ is stabilizable if and only if  there exist $P\in {\mathbb S}^{(n+m)}_{++}$ such that $A_F P A_F^T \prec P$. Replacing $P$ with $SPS$ leads to the equivalent condition $A_F S(F) P S(F) A_F^T \prec S(F)PS(F)$. Using the relation in~\cref{lemma:S-H2} yields the conclusion.
\end{proof}

\cref{proposition:stab-eval} will be useful when we want to check if the unknown system is (asymptotically) stable. Its main feature is that it requires the on-policy data from~\cref{algo:data-collection2}. Note however that it does not require the knowledge of $F$.
Next, a stabilizing state-feedback control design algorithm is proposed using LMIs and data from~\cref{algo:data-collection1}.
\begin{proposition}[Stabilization]\label{proposition:stab-design}
The system~\eqref{eq:LTI-system} is stabilizable if and only if there exist $G\in {\mathbb R}^{n \times n}$, $P \in {\mathbb S}^{n+m}$, and ,$X \in {\mathbb R}^{n \times m}$, such that the following LMI holds:
\begin{align}
\left[ {\begin{array}{*{20}c}
   { - SPS} & *  \\
   {\left[ {\begin{array}{*{20}c}
   G & X  \\
\end{array}} \right]} & {H^T PH - G - G^T }  \\
\end{array}} \right] \prec 0\label{eq:1}
\end{align}
If a solution, $(\bar P,\bar G,\bar X)$, exists, then a stabilizing state-feedback gain is given by  $F = \bar X^T (\bar G^T )^{ - 1}$, and $V(x) = x^T S\bar PSx$ is the corresponding Lyapunov function of $A_F$.
\end{proposition}
\begin{proof}
\cref{lemma:spetral-radius-lemma} tells us that the original system~\eqref{eq:LTI-system} is stabilizable if and only if the augmented system~\eqref{eq:augmented-system} is stabilizable, or equivalently $A_F$ is Schur. Moreover, from a standard result of the linear system theory, we know that $A_F$ is Schur if and only if the corresponding dual system $A_F^T$ is Schur. From the Lyapunov theory, the dual system $A_F^T$ is Schur if and only if there exists a Lyapunov matrix $P\in {\mathbb S}^{n+m}_{++}$ such that $A_F PA_F^T  \prec P$. The Lyapunv inequality can be expressed as
\begin{align*}
&\left[ {\begin{array}{*{20}c}
   I  \\
   {\left[ {\begin{array}{*{20}c}
   I & {F^T }  \\
\end{array}} \right]}  \\
\end{array}} \right]^T \left[ {\begin{array}{*{20}c}
   { - P} & 0  \\
   0 & {\left[ {\begin{array}{*{20}c}
   {A^T }  \\
   {B^T }  \\
\end{array}} \right]^T P\left[ {\begin{array}{*{20}c}
   {A^T }  \\
   {B^T }  \\
\end{array}} \right]}  \\
\end{array}} \right]\\
&\times \left[ {\begin{array}{*{20}c}
   I  \\
   {\left[ {\begin{array}{*{20}c}
   I & {F^T }  \\
\end{array}} \right]}  \\
\end{array}} \right] \prec 0
\end{align*}

From~\cref{lemma:Finsler} (Finsler lemma), we have that dual system $A_F^T$ is Schur if and only if there exist $P,F,G$ such that
\[
\left[ {\begin{array}{*{20}c}
   { - P} & *  \\
   {G\left[ {\begin{array}{*{20}c}
   I & {F^T }  \\
\end{array}} \right]} & {\left[ {\begin{array}{*{20}c}
   {A^T }  \\
   {B^T }  \\
\end{array}} \right]^T P\left[ {\begin{array}{*{20}c}
   {A^T }  \\
   {B^T }  \\
\end{array}} \right] - G - G^T }  \\
\end{array}} \right] \prec 0
\]
which is a non-convex bilinear matrix inequality. The first block diagonal matrix ensures $P \succ 0$, and the second block diagonal matrix implies $G + G^T \succ 0$. This guarantees that $G$ is nonsingular. With the change of variables, $X  = GF^T$, the last matrix inequality becomes
\begin{align}
\left[ {\begin{array}{*{20}c}
   { - P} & *  \\
   {\left[ {\begin{array}{*{20}c}
   G & X  \\
\end{array}} \right]} & {\left[ {\begin{array}{*{20}c}
   {A^T }  \\
   {B^T }  \\
\end{array}} \right]^T P\left[ {\begin{array}{*{20}c}
   {A^T }  \\
   {B^T }  \\
\end{array}} \right] - G - G^T }  \\
\end{array}} \right] \prec 0\label{eq:2}
\end{align}
Clearly, the above linear matrix inequality~\eqref{eq:2} holds if and only if the previous bilinear matrix inequality is satisfied from the bijective mapping $F^T = G^{-1}X$. Next, we replace $P$ with $SPS$ and use the identity~\cref{lemma:S-H} to obtain the LMI~\eqref{eq:1} in the statement. Note that~\eqref{eq:2} holds if and only if~\eqref{eq:1} because $S \in {\mathbb S}^{n+m}_{++}$ is nonsingular.
This completes the proof.
\end{proof}

Using the LMI condition in~\cref{proposition:stab-design}, a stabilizing state-feedback controller can be found only using the trajectories. Note that the data used in~\cref{proposition:stab-design} is generated from the off-policy method~\cref{algo:data-collection2}.

\section{Data-driven LMIs for LQR design}

Beyond the stabilization problem, the idea in the previous section can be also applied to LQR design problems.
We first consider a policy evaluation problem again. Given a potentially unknown state-feedback gain $F$, suppose that we only have an access to the state-input trajectories. Under this situation, the problem is to determine the LQR performance of the unknown $F$.
\begin{proposition}[Performance evaluation]\label{thm:LQR-evaluation}
Consider the optimization problem
\begin{align}
&\mathop {{\rm{min}}}\limits_{P \in {\mathbb S}^{n + m} } \,\,{\bf Tr}\left( {\Lambda S(F)PS(F)} \right)\label{eq:4}\\
&{\rm subject}\,\,{\rm{to}}\quad H(F)^T PH(F) + I \preceq S(F)PS(F)\nonumber
\end{align}
and $\bar P\in {\mathbb S}^{n+m}$ is the corresponding optimal point. Then, the optimal objective function value~\eqref{eq:4} is the cost corresponding to $F$, i.e., ${\bf Tr}\left( {\Lambda S(F)\bar PS(F)} \right) = J(F)$.
\end{proposition}
\begin{proof}
The optimal solution $\bar P$ satisfies
\[
H(F)^T \bar PH(F) + I \preceq S(F) \bar PS(F),\quad P \succ 0.
\]
Using~\cref{lemma:S-H2} and letting $\tilde P = S(F) \bar PS(F)$, it follows that
\begin{align}
A_F \tilde PA_F^T  + I \preceq \tilde P,\quad \tilde P \succ 0.\label{eq:9}
\end{align}

Since the above inequality is a Lyapunov inequality, $A_F$ is Scuhr. Therefore, there exists $\hat P \in {\mathbb S}_{++}$ such that $A_F \hat PA_F^T  + I = \hat P$, where $\hat P: = \sum_{k = 0}^\infty  {A_F^k (A_F^T )^k }$. Replacing $\hat P$ with $S(F) M S(F)$, where $M = S(F)^{-1} \hat P S(F)^{-1}$, we can see that $M$ satisfies $H(F)^T MH(F) + I = S(F)MS(F)$. This implies that $M$ is a feasible point for~\eqref{eq:4}. Therefore,
\begin{align*}
{\bf Tr}\left( {\Lambda S(F)\bar PS(F)} \right) \leq & {\bf Tr}(\Lambda S(F)MS(F))= {\bf Tr} ( \Lambda \hat P)
\end{align*}
On the other hand, repeatedly applying the inequality~\eqref{eq:9} yields
\[
\sum\limits_{k = 0}^\infty  {A_k^k (A_F^T )^k } \preceq \tilde P = S(F)\bar PS(F)
\]
by which we have ${\bf Tr}(\Lambda S(F)\bar PS(F)) \ge {\bf Tr}(\Lambda \hat P)$, implying ${\bf Tr} (\Lambda S(F)\bar PS(F)) = {\bf Tr} (\Lambda \hat P)$.
Then, we can conclude
\begin{align*}
{\bf Tr}\left( {\Lambda S(F)\bar PS(F)} \right)=& {\bf Tr}(\Lambda \hat P)\\
=& {\bf Tr} \left( {\sum\limits_{k = 0}^\infty  {(A_F^T )^k \Lambda A_F^k } } \right)\\
=& \sum\limits_{i = 1}^n {\sum\limits_{k = 0}^\infty  {e_i^T (A_F^T )^k \Lambda A_F^k e_i } }\\
=& J(F)
\end{align*}
This completes the proof.
\end{proof}

Next, the LQR design problem is addressed using a data-driven LMI. The following LMI condition allows us to design an LQR control of unknown system in a simple and efficient way.
\begin{proposition}[LQR design]\label{thm:LQR-design}
Consider the optimization problem
\begin{align}
&\mathop {{\rm{min}}}_{P \in {\mathbb S}^{n + m} ,\,G \in {\mathbb R}^{n \times n} ,\,X \in {\mathbb R}^{n \times m} } \,\,{\bf Tr}\left( {\Lambda SPS} \right)\label{eq:5}\\
&{\rm{subject}}\,\,{\rm{to}}\,\,\,\left[ {\begin{array}{*{20}c}
   { - SPS + I} & *  \\
   {\left[ {\begin{array}{*{20}c}
   G & X  \\
\end{array}} \right]} & {H^T PH - G - G^T }  \\
\end{array}} \right] \prec 0\nonumber
\end{align}
and $\bar G\in {\mathbb R}^{n \times n}$, $\bar P \in {\mathbb S}^{n+m}$, and ,$\bar Y \in {\mathbb R}^{n \times m}$ are the corresponding optimal points. Then, the optimal objective function value upper bounds the optimal cost, $J(F^*)$, and the corresponding state-feedback gain is given by $\bar F = \bar X^T (\bar G^T )^{ - 1}$,
\end{proposition}
\begin{proof}
The LMI constraint is identical to the stabilization case. Therefore, we can follow the same procedure to arrive that the conclusion that the optimization is equivalent to the following optimization:
\begin{align}
&\mathop {{\rm{min}}}_{P \in {\mathbb S}^{n + m} ,\,F \in {\mathbb R}^{m \times n} } \,\, {\bf Tr} \left( {\Lambda P} \right)\label{eq:3}\\
&{\rm{subject}}\,\,{\rm{to}}\quad \left[ {\begin{array}{*{20}c}
   I  \\
   F  \\
\end{array}} \right]\left( {\left[ {\begin{array}{*{20}c}
   {A^T }  \\
   {B^T }  \\
\end{array}} \right]^T P\left[ {\begin{array}{*{20}c}
   {A^T }  \\
   {B^T }  \\
\end{array}} \right]} \right)\left[ {\begin{array}{*{20}c}
   I  \\
   F  \\
\end{array}} \right]^T + I \prec P\nonumber
\end{align}
Therefore, $(S\bar P S, \bar F)$ is an optimal solution of the above problem.
Applying the inequality recursively leads to
\[
\sum_{k = 0}^{N - 1} {A_{\bar F}^k (A_{\bar F}^T )^k } \preceq A_{\bar F}^{N - 1} S \bar P S(A_{\bar F}^T )^{N - 1}  + \sum\limits_{k = 0}^{N - 1} {A_{\bar F}^k (A_{\bar F}^T )^k }  \prec S \bar P S
\]

Multiplying with $\Lambda$ and taking the trace on the last inequality, one gets
\begin{align*}
{\bf Tr}(\Lambda S \bar P S) \ge& {\bf Tr}\left( {\Lambda \sum_{k = 0}^{N - 1} {A_{\bar F}^k (A_{\bar F}^T )^k } } \right)\\
=& \sum_{i = 1}^n {\sum_{k = 0}^{N - 1} {x(k;\bar F, e_i )^T \Lambda x(k;\bar F,e_i )} }\\
 = & J(\bar F).
\end{align*}
Taking the limit $N \to \infty$, we obtain the desired conclusion.
\end{proof}

\cref{thm:LQR-design} allows us to design a controller with a guaranteed upper bound on the LQR performance. A natural question arising here is whether or not the obtained controller from~\cref{thm:LQR-design} is optimal. If not, then how far is it away from the optimal gain $F^*$? A potential answer is given in the following result. In particular, to answer this question, one needs to make it clear that the LMI in~\cref{thm:LQR-design} is strict. Note that the LMI needs the strictness to use the Finsler's lemma. Therefore, the feasible set satisfying the LMI constraint is an open set, and therefore, there would be no solution to the optimization in~\cref{thm:LQR-design}. In practice, to find an approximate solution to~\cref{thm:LQR-design}, most LMI solvers try to solve the semi-definite problem with a small margin. For simplicity and convenience, let us start with~\eqref{eq:3}, and consider the modified problem
\begin{align}
&\mathop {{\rm{min}}}_{P \in {\mathbb S}^{n + m} ,\,F \in {\mathbb R}^{m \times n} } \,\,{\bf{Tr}}\left( {\Lambda SPS} \right)\label{eq:6}\\
&{\rm{subject}}\,\,{\rm{to}}\quad \left[ {\begin{array}{*{20}c}
   I  \\
   F  \\
\end{array}} \right]\left( {\left[ {\begin{array}{*{20}c}
   {A^T }  \\
   {B^T }  \\
\end{array}} \right]^T SPS \left[ {\begin{array}{*{20}c}
   {A^T }  \\
   {B^T }  \\
\end{array}} \right]} \right)\left[ {\begin{array}{*{20}c}
   I  \\
   F  \\
\end{array}} \right]^T + I \nonumber \\
& \preceq - \varepsilon I + SPS \nonumber
\end{align}
with a sufficiently small $\varepsilon>0$. Define solution $(G,P,Y) = (\bar G_\varepsilon, \bar P_\varepsilon, \bar Y_\varepsilon)$ to~\eqref{eq:6}. We characterize the solution to~\eqref{eq:5} and equivalently~\eqref{eq:6} as $(G,P,Y) = (\bar G_\varepsilon, \bar P_\varepsilon, \bar Y_\varepsilon)$ in the limit $\varepsilon \to 0$. Based on this definition, we can obtain an optimality of the solution to~\cref{thm:LQR-design}. Indeed, we prove that the feedback gain $\bar F_\varepsilon$ is optimal for any $\varepsilon > 0$.
\begin{proposition}
Suppose that $(G,P,Y) = (\bar G_\varepsilon, \bar P_\varepsilon, \bar Y_\varepsilon)$ is a solution of~\eqref{eq:6}, and let $\bar F_\varepsilon = \bar X_\varepsilon^T (\bar G_\varepsilon^T )^{ - 1}$.
The feedback gain $\bar F_\varepsilon$ is optimal for any $\varepsilon > 0$.
\end{proposition}
\begin{proof}
Plugging $(\bar G_\varepsilon, \bar P_\varepsilon, \bar Y_\varepsilon)$ into $(G,P,Y)$ in the constraint~\eqref{eq:6}, we have
\begin{align}
&\left[ {\begin{array}{*{20}c}
   I  \\
   {\bar F_\varepsilon  }  \\
\end{array}} \right]\left( {\left[ {\begin{array}{*{20}c}
   {A^T }  \\
   {B^T }  \\
\end{array}} \right]^T S\bar P_\varepsilon  S\left[ {\begin{array}{*{20}c}
   {A^T }  \\
   {B^T }  \\
\end{array}} \right]} \right)\left[ {\begin{array}{*{20}c}
   I  \\
   {\bar F_\varepsilon  }  \\
\end{array}} \right]^T+ (1+ \varepsilon) I \nonumber\\
\preceq & S\bar P_\varepsilon  S,\label{eq:7}
\end{align}
which is a Lyapunov inequality. Therefore, $A_{\bar F_{\varepsilon}}$ is Schur, and by the Lyapunov theory, there exists a Lyapunov matrix
\[
S \hat P_\varepsilon S  : = (1 + \varepsilon) \sum\limits_{k = 0}^\infty  {A_{\bar F_\varepsilon  }^k (A_{\bar F_\varepsilon  }^T )^k }
\]
such that $A_{\bar F} S \hat P_\varepsilon S  A_{\bar F}^T  + (1+\varepsilon) I = S \hat P_\varepsilon S$. Obviously, $\hat P_\varepsilon$ is a feasible solution to~\eqref{eq:6}, and hence, from the optimality of $\bar P_\varepsilon$, it holds that ${\bf Tr}(\Lambda S\bar P_\varepsilon S) \leq {\bf Tr}(\Lambda S \hat P_\varepsilon S )=(1 + \varepsilon )J(\bar F_\varepsilon  )$. On the other hand, recursively applying~\eqref{eq:7} leads to
\[
S\bar P_\varepsilon  S \succeq (1 + \varepsilon )\sum\limits_{k = 0}^\infty  {A_{\bar F_\varepsilon  }^k (A_{\bar F_\varepsilon  }^T )^k }
\]
and hence, ${\bf Tr}\left( {\Lambda S\bar P_\varepsilon  S} \right) \ge {\bf Tr}(\Lambda S\hat P_\varepsilon  S)=(1 + \varepsilon )J(\bar F_\varepsilon)$. Combining the last two inequalities, we have ${\bf Tr}(\Lambda S\bar P_\varepsilon  S )=(1 + \varepsilon )J(\bar F_\varepsilon)$. By contradiction, assume that there exists an optimal feedback gain $F^* $ such that
\begin{align}
&(1 + \varepsilon )J(\bar F_\varepsilon  ) > (1 + \varepsilon )J(F^* )\label{eq:8}
\end{align}
Then. there exists
\[
SP^*S  : = (1+\varepsilon)\sum\limits_{k = 0}^\infty  {A_{F^* }^k (A_{F^* }^T )^k }
\]
such that $A_{F^* } SP^*S A_{F^* }^T  + (1+\varepsilon)I = SP^*S$. Since $(P,F) = (P^*, F^*)$ is feasible solution to~\eqref{eq:6}, we have ${\bf Tr}(\Lambda S\bar P_\varepsilon  S) \le {\bf Tr}( \Lambda SP^* S)= (1 + \varepsilon) J(F^*)$. Combining the last inequality with~\eqref{eq:8}, we arrive at a contradiction. Therefore, $\bar F_\varepsilon$ is the optimal feedback gain for any $\varepsilon$. This completes the proof.
\end{proof}

\section{Data-driven dynamic programming}
Although the data-driven LMIs in the previous sections are efficient, it is still meaningful to briefly discuss and summarize dynamic programming methods~\cite{bertsekas1996neuro}, which does not depend on LMI solvers. The previous ideas can be extended to dynamic programming summarized in~\cref{algo:policy-iteration} and~\cref{algo:value-iteration}.
\begin{algorithm}[h]
\caption{Data-Driven Policy Iteration}
\begin{algorithmic}[1]
\State Initialize $F_0 = 0$.
\For{$k \in \{0,1,\ldots\}$}
\State Collect data $(S(F_k),H(F_k))={\tt On-Collect}(F_k)$
\State Solve for $P_{k+1}$ the linear equation
\[
H(F_k )^T P_{k+1} H(F_k ) + S(F_k )\Lambda S(F_k ) = S(F_k )P_{k+1} S(F_k )
\]

\State Update $F_{k + 1}  =  - P_{k+1,22}^{ - 1} P_{k+1,12}^T$

\If{$\left\| {P_k  - P_{k + 1} } \right\| \le \varepsilon$}
\State Stop and return $P_{k+1}$ and $F_{k + 1}  =  - P_{k + 1,22}^{ - 1} P_{k + 1,12}^T$
\EndIf

\EndFor

\end{algorithmic}\label{algo:policy-iteration}
\end{algorithm}
\cref{algo:policy-iteration} summarizes a policy iteration algorithm proposed in~\cite{lee2018primal} for completeness. Its convergence was also proved in~\cite{lee2018primal}.
\begin{proposition}[Convergence of~\cref{algo:policy-iteration}, \cite{lee2018primal}]
The iteration $P_k$ in~\cref{algo:policy-iteration} converges to $P^*$ defined in~\eqref{eq:P*}.
\end{proposition}

The main feature of~\cref{algo:policy-iteration} is that it uses on-policy data generated by~\cref{algo:data-collection2}.
Therefore, it needs to collect new data at every iterations, and each data collection should apply the exploring starts scheme.
The newly proposed value iteration algorithm presented in~\cref{algo:value-iteration} suggests an off-policy algorithm in the sense that the policy used to generate the data is independent of the policy we want to learn or the intermediate policies while learning. Therefore, it collects data once at the beginning. Moreover, it does not need to stick to the exploring starts scheme because the exploratory inputs can be used during the data collection. In this sense, the new~\cref{algo:value-iteration} is more sample efficient than~\cref{algo:policy-iteration}.
\begin{algorithm}[h]
\caption{Data-Driven Value Iteration}
\begin{algorithmic}[1]
\State Initialize $P_0  = 0$.
\State Given fixed initial state $x(0) = z $, collect data $(S,H)={\tt Off-Collect}(z)$
\For{$k \in \{0,1,\ldots\}$}

\State Solve for $P_{k + 1}$ the linear matrix equation
\[
S P_{k + 1} S  =S \Lambda S  + H(P_{k,11}  - P_{k,12} P_{k,22}^{ - 1} P_{k,12}^T )H^T
\]

\If{$\left\| {P_k  - P_{k + 1} } \right\| \le \varepsilon$}
\State Stop and return $P_{k+1}$ and $F_{k + 1}  =  - P_{k + 1,22}^{ - 1} P_{k + 1,12}^T$
\EndIf

\EndFor

\end{algorithmic}\label{algo:value-iteration}
\end{algorithm}
Multiplying both sides of the linear matrix equation in~\cref{algo:value-iteration} by $S$, it is reduced to
\[
P_{k + 1}  =\Lambda  + S^{ - 1} H(P_{k,11}  - P_{k,12} P_{k,22}^{ - 1} P_{k,12}^T )H^T S^{ - 1}
\]
which can be interpreted as a model-based value iteration because $H^T S^{ - 1}  = \left[ {\begin{array}{*{20}c}
   A & B  \\
\end{array}} \right]$ from~\cref{lemma:S-H}. Lastly, we establish the convergence of~\cref{algo:value-iteration}.
\begin{proposition}[Convergence of~\cref{algo:value-iteration}]
The iteration $P_k$ in~\cref{algo:value-iteration} converges to $P^*$.
\end{proposition}
\begin{proof}
We only need to prove that~\cref{algo:value-iteration} is equivalent to the Q-value iteration, which is known to converge to the optimal $P^*$~\cite{bertsekas1996neuro}. Applying~\cref{lemma:S-H} and multiplying both sides of the $P$-update equation in~\cref{algo:value-iteration} by $S^{-1}$, we obtain
\begin{align*}
P_{k + 1}  =& \Lambda+ \left[ {\begin{array}{*{20}c}
   A & B  \\
   { - P_{k,22}^{ - 1} P_{k,12} A} & { - P_{k,22}^{ - 1} P_{k,12} B}  \\
\end{array}} \right]^T\\
& \times P_k \left[ {\begin{array}{*{20}c}
   A & B  \\
   { - P_{k,22}^{ - 1} P_{k,12} A} & { - P_{k,22}^{ - 1} P_{k,12} B}  \\
\end{array}} \right]
\end{align*}
Multiplying both sides by $\left[ {\begin{array}{*{20}c}
   x  \\
   u  \\
\end{array}} \right]$ from the right and its transpose from the left, we have
\[
Q_{k + 1} (x,u) = \left[ {\begin{array}{*{20}c}
   x  \\
   u  \\
\end{array}} \right]^T \Lambda \left[ {\begin{array}{*{20}c}
   x  \\
   u  \\
\end{array}} \right] + \min _{v \in R^n } Q_k (x,v)
\]
with $Q_k (x,u) = \left[ {\begin{array}{*{20}c}
   x  \\
   u  \\
\end{array}} \right]^T P_k \left[ {\begin{array}{*{20}c}
   x  \\
   u  \\
\end{array}} \right]$. It is equivalent to the Q-value iteration~\cite{bertsekas2005dynamic}, which is known to converge to $P^*$, where $Q^* (x,u) = \left[ {\begin{array}{*{20}c}
   x  \\
   u  \\
\end{array}} \right]^T P^* \left[ {\begin{array}{*{20}c}
   x  \\
   u  \\
\end{array}} \right]$. This completes the proof.
\end{proof}

\section{Exploration schemes}

For the on-policy data collection,~\cref{algo:data-collection2}, the exploring starts always guarantee $S(F) \succ 0$. However, collecting the trajectories with different initial points which span ${\mathbb R}^n$ may not be tractable in practice. The off-policy data collection,~\cref{algo:data-collection1}, is relevantly more promising in this respect, because it can use the exploratory inputs while generating the trajectories, and can be used in the case that the initial state is given and fixed. We can apply an arbitrary inputs, $u(k)$, and expect that $S \succ 0$ eventually under the persistent excitation assumption. A standard exploration strategy is to inject the i.i.d. Gaussian noises, $u(k) \sim {\cal N}(0,U)$, where $U \in {\mathbb S}^{m}_{++}$ is the covariance matrix. If trajectories starting from the fixed $x(0) =z$ can be collected as many as possible, then we can develop a new version of the off-policy exploration strategy given in~\cref{algo:data-collection3}, which offers theoretical guarantees of the data validity under a mild assumption, i.e., the controllability.
\begin{algorithm}[h]
\caption{Off-policy data collection $(S,H)={\tt Off-Collect2}(z)$ with restarting}
\begin{algorithmic}[1]
\State Initialize $S_0  = 0,H_0  = 0$.

\For{$i \in \{1,2,\ldots, N\}$}
\State Initialize $x(0;i)=z$.
\State Initialize $\tilde S_{0;i}  = 0,\tilde H_{0;i}  = 0$.
\For{$k \in \{0,1,\ldots, n-1\}$}
\State Apply control input $u(k;i)=\zeta(k;i), \zeta(k;i) \sim {\cal N}(0,U)$
\State Observe $x(k+1;i)$

\State Update
\[
\tilde S_{k + 1;i}  \leftarrow \tilde S_{k;i}  + \left[ {\begin{array}{*{20}c}
   {x(k;i)}  \\
   {u(k;i)}  \\
\end{array}} \right]\left[ {\begin{array}{*{20}c}
   {x(k;i)}  \\
   {u(k;i)}  \\
\end{array}} \right]^T
\]
\[
\tilde H_{k + 1;i}  \leftarrow \tilde H_{k;i}  + \left[ {\begin{array}{*{20}c}
   {x(k;i)}  \\
   {u(k;i)}  \\
\end{array}} \right]x(k + 1;i)^T
\]

\EndFor

\State Update
\[
S_{i + 1}  \leftarrow \frac{i}{{i + 1}} S_i  + \frac{1}{{i + 1}} \tilde S_{n;i}
\]
\[
H_{i + 1}  \leftarrow \frac{i}{{i + 1}} H_i  + \frac{1}{{i + 1}} \tilde H_{n;i}
\]

\EndFor

\State Return $(S,H) = (S_{N}, H_{N})$

\end{algorithmic}\label{algo:data-collection3}
\end{algorithm}

In~\cref{algo:data-collection3}, $N$ trajectories are collected and then averaged, i.e., $S_N  = \frac{1}{N}\sum_{i = 1}^N {\tilde S_{n;i} } ,H_N  = \frac{1}{N}\sum_{i = 1}^N {\tilde H_{n;i} }$. Each trajectory starts from $x(0) = z$ which is fixed. We can readily prove that the data matrices from~\cref{algo:data-collection3} also satisfies the data transformation property~\cref{lemma:S-H}. We can also prove that if $(A,B)$ is controllable, then the data collection strategy guarantees that $S_{N}$ converges to a strictly positive definite matrix with probability one as $N \to \infty$.
\begin{theorem}\label{thm:exploration}
Suppose that $(A,B)$ is controllable, and consider~\cref{algo:data-collection3}, whose output is
\[
S_N  = \frac{1}{N}\sum\limits_{i = 1}^N {\left( {\sum\limits_{k = 0}^n {\left[ {\begin{array}{*{20}c}
   {x(k;i)}  \\
   {u(k;i)}  \\
\end{array}} \right]\left[ {\begin{array}{*{20}c}
   {x(k;i)}  \\
   {u(k;i)}  \\
\end{array}} \right]^T } } \right)}
\]
where $x(k;i)$ and $u(k;i)$ stand for the state and input at time $k$ at the $i$th outer iteration.
Then, we have
\[
{\mathbb P}\left[ {\mathop {\lim }\limits_{N \to \infty } S_N \succ 0} \right] = 1
\]
\end{theorem}
\begin{proof}
Define
\begin{align}
{\cal O}_k : =& \left[ {\begin{array}{*{20}c}
   B & {AB} &  \cdots  & {A^{k - 1} B}  \\
\end{array}} \right]\nonumber\\
{\cal U}_k : =& {\rm{diag}}(\underbrace {U, \ldots ,U}_{\rm k-times})\label{eq:13}
\end{align}
and
\begin{align}
u_{k;i} : = \left[ {\begin{array}{*{20}c}
   {\zeta(k - 1;i)}  \\
    \vdots   \\
   {\zeta(1;i)}  \\
   {\zeta(0;i)}  \\
\end{array}} \right]\label{eq:12}
\end{align}
Then, $x(k;i)$ is expressed as $x(k;i) = A^k z + {\cal O}_k u_{k;i}$, and thus
\[
x(k;i)x(k;i)^T  = A^k z^T z (A^T )^k  + 2A^k zu_{k;i}^T {\cal O}_k^T  + {\cal O}_k u_{k;i} u_{k;i}^T {\cal O}_k^T.
\]
Taking the expectation leads to
\[
{\mathbb E}[x(k;i)x(k;i)^T] = A^k z z^T (A^T )^k  + {\cal O}_k {\cal U}_k {\cal O}_k^T
\]
At $k = n$, ${\cal O}_n$ is the controllability matrix, and it is full row rank due to the controllability in~\cref{assumption:basic-assumption}. Since ${\cal U}_k \succ 0$, one concludes ${\mathbb E}[x(n)x(n)^T] \succ 0$.
Since $u(k;i)$ is i.i.d. and the initial state is reset periodically after $n$ steps, $S_{N}$ is written as
\[
S_N  = \frac{1}{N}\sum\limits_{i = 1}^N {\sum\limits_{k = 0}^n {\left[ {\begin{array}{*{20}c}
   {x(k;i)}  \\
   {u(k;i)}  \\
\end{array}} \right]\left[ {\begin{array}{*{20}c}
   {x(k;i)}  \\
   {u(k;i)}  \\
\end{array}} \right]^T } }  = \frac{1}{N}\sum\limits_{i = 1}^N {M_i }
\]
where
\[
M_i  = \sum\limits_{k = 0}^n {\left[ {\begin{array}{*{20}c}
   {x(k;i)}  \\
   {u(k;i)}  \\
\end{array}} \right]\left[ {\begin{array}{*{20}c}
   {x(k;i)}  \\
   {u(k;i)}  \\
\end{array}} \right]^T }
\]
is an i.i.d. random variables with mean
\begin{align*}
&{\mathbb E}[M_i ] = M:=\sum\limits_{k = 0}^n \left[ {\begin{array}{*{20}c}
   {A^k z z^T (A^T )^k  + {\cal O}_k {\cal U}_k {\cal O}_k^T } & 0  \\
   0 & U  \\
\end{array}} \right]\succ  0
\end{align*}
By the strong law of large numbers, we get $
{\mathbb P}\left[ \lim_{N \to \infty } S_{N}  = M \right] = 1$, which leads to the desired conclusion.
\end{proof}

\cref{algo:data-collection3} provides a data collection scheme with theoretical guarantees of the validity of the data. It is useful especially when the exploring starts scheme (starting with arbitrary initial states) is not available. However, it still requires the ability to generate $N$ trajectories from the given initial state $z$. In practice, if only a single trajectory starting from a fixed $z$ is available, we can develop another data acquisition method given in~\cref{algo:data-collection4}. The benefit comes from some cost to pay. In particular, it initially needs a stabilizing state-feedback gain $K$ or at least, the system $A$ itself needs to be stable. In such case, we can approximately mimic the restarting strategy in~\cref{algo:data-collection3} using the stability of the closed-loop system $A+BK$. \cref{algo:data-collection4} will be called the off-policy data collection with periodic excitation.
\begin{algorithm}[h]
\caption{Off-policy data collection $(S,H)={\tt Off-Collect2}(z)$ with periodic excitation}
\begin{algorithmic}[1]
\State Initialize $S_0  = 0,H_0  = 0$.
\State Initialize state $x(0)=z$.

\State Initialize $\varepsilon >0$

\For{$i \in \{1,2,\ldots,N\}$}
\State Initialize time $k=0$
\While{$\left\| { x(k)} \right\| > \varepsilon$}
\State Apply control input $\tilde u(k) = K x(k)$
\State $k \leftarrow k+1$
\EndWhile

\State Initialize $\tilde S_0  = 0,\tilde H_0  = 0$.

\State Initialize time $k=0$ and $x(0;i):= x(0)$
\For{$k \in \{0,1,\ldots,n-1\}$}
\State Apply control input $u(k;i)= Kx(k;i) + \zeta(k;i), \zeta(k;i) \sim {\cal N}(0,U)$
\State Observe $x(k+1;i):= x(k+1)$
\State Update
\[
\tilde S_{k + 1}  \leftarrow \tilde S_k  + \left[ {\begin{array}{*{20}c}
   {x(k;i)}  \\
   {u(k;i)}  \\
\end{array}} \right]\left[ {\begin{array}{*{20}c}
   {x(k;i)}  \\
   {u(k;i)}  \\
\end{array}} \right]^T
\]
\[
\tilde H_{k + 1}  \leftarrow \tilde H_k  + \left[ {\begin{array}{*{20}c}
   {x(k;i)}  \\
   {u(k;i)}  \\
\end{array}} \right]x(k + 1;i)^T
\]
\EndFor

\State Update
\[
S_{i + 1}  \leftarrow \frac{i}{{i + 1}} S_i  + \frac{1}{{i + 1}} \tilde S_n
\]
\[
H_{i + 1}  \leftarrow \frac{i}{{i + 1}} H_i  + \frac{1}{{i + 1}} \tilde H_n
\]

\EndFor

\State Return $(S,H) = (S_N, H_N)$

\end{algorithmic}\label{algo:data-collection4}
\end{algorithm}
The main feature of~\cref{algo:data-collection4} lies in that the process can be interpreted as an alternation of the two phases: the first phase is a settling down period, where the state tends to vanish without the excitation signals in the input $u(k)$. This phase stops when the current state $x(k)$ is sufficiently small in the sense that $\|x(k)\| \le \varepsilon$ for a sufficiently small $\varepsilon>0$. The second phase is an excitation or exploration period, where the state is excited by injecting Gaussian noises in the input. As in~\cref{thm:exploration}, we can prove that~\cref{algo:data-collection4} theoretically ensures the validity of the data output provided that $(A,B)$ is controllable.
\begin{theorem}\label{thm:exploration2}
Suppose that $(A,B)$ is controllable, and consider~\cref{algo:data-collection3}, whose output is
\[
S_N  = \frac{1}{N}\sum\limits_{i = 1}^N {\left( {\sum\limits_{k = 0}^n {\left[ {\begin{array}{*{20}c}
   {x(k;i)}  \\
   {u(k;i)}  \\
\end{array}} \right]\left[ {\begin{array}{*{20}c}
   {x(k;i)}  \\
   {u(k;i)}  \\
\end{array}} \right]^T } } \right)}
\]
where $x(k;i)$ and $u(k;i)$ stand for the state and input, respectively, at time $k$ at the $i$th outer iteration.
Then, there exists a sufficient small $\varepsilon>0$ such that
\[
{\mathbb P}\left[ {\mathop {\lim }\limits_{N \to \infty } S_{N} \succ 0} \right] = 1
\]
\end{theorem}
\begin{proof}
Define
\begin{align*}
{\cal O}_k : =& \left[ {\begin{array}{*{20}c}
   B & {(A+BK)B} &  \cdots  & {(A+BK)^{k - 1} B}  \\
\end{array}} \right].
\end{align*}

Then, the state at time $k$ is $x(k;i) = (A + BK)^k z_i  + {\cal O}_{k} u_{k;i}$, where $z_i$ is the initial state, $x(0;i) = z_i$ at the $i$th period such that $\left\| {z_i } \right\| \le \varepsilon$, and $u_{k;i}$ is defined in~\eqref{eq:12}. Then, one gets
\begin{align}
x(k;i)x(k;i)^T  =& (A + BK)^k z_i z_i^T ((A + BK)^T )^k\nonumber \\
&+ 2(A + BK)^k z_i u_{k;i}^T {\cal O}_k^T\label{eq:10} \\
& + {\cal O}_k u_{k;i} u_{k;i}^T {\cal O}_k^T\nonumber
\end{align}
which is lower bounded by
\begin{align*}
x(k;i)x(k;i)^T\succeq & (A + BK)^k z_i z_i^T ((A + BK)^T )^k\\
& - (A + BK)^k z_i z_i^T ((A + BK)^T )^k \frac{1}{\varepsilon } \\
&- \varepsilon {\cal O}_k u_{k;i} u_{k;i}^T {\cal O}_k^T  + {\cal O}_k u_{k;i} u_{k;i}^T {\cal O}_k^T
\end{align*}
where~\cref{lemma:bounding-lemma} was applied to~\eqref{eq:10}. Again, the last bound is further bounded from below as
\begin{align*}
&x(k;i)x(k;i)^T\\
\succeq& (A + BK)^k z_i z_i^T ((A + BK)^T )^k\\
&- (A + BK)^k z_i z_i^T ((A + BK)^T )^k \frac{1}{\varepsilon }\\
& - \varepsilon {\cal O}_n u_{k;i} u_{k;i}^T O_k^T  + {\cal O}_k u_{k;i} u_{k;i}^T {\cal O}_k^T\\
\succeq& (A + BK)^k z_i z_i^T ((A + BK)^T )^k\\
&- I\lambda _{\max } ((A + BK)^k z_i z_i^T ((A + BK)^T )^k )\frac{1}{\varepsilon }\\
& + (1 - \varepsilon ){\cal O}_k u_{k;i} u_{k;i}^T {\cal O}_k^T\\
\succeq& (A + BK)^k z_i z_i^T ((A + BK)^T)^k\\
& - I\|(A + BK)^k\|^2 \|z_i\|^2 \frac{1}{\varepsilon} + (1 - \varepsilon ){\cal O}_k u_{k;i} u_{k;i}^T {\cal O}_k^T\\
\succeq& (A + BK)^k z_i z_i^T ((A + BK)^T )^k \\
& - \varepsilon I \|(A + BK)^k\|^2 + (1 - \varepsilon ){\cal O}_k u_{k;i} u_{k;i}^T {\cal O}_k^T
\end{align*}
where $\lambda_{\max} (\cdot)$ denotes the maximum eigenvalue of a symmetric matrix, and the last inequality uses the fact that $\|z_i\| \le \varepsilon$.

On the other hand, noting $x(k;i)u(k;i)^T  = (A + BK)^k z_i u(k)^T  + {\cal O}_k u_k u(k)^T$, we have
\begin{align*}
&\left[ {\begin{array}{*{20}c}
   0 & {x(k;i)u(k;i)^T }  \\
   {u(k;i) x(k;i)^T } & 0  \\
\end{array}} \right]\\
=& \left[ {\begin{array}{*{20}c}
   0 & {(A + BK)^k z_i u(k;i)^T }  \\
   {u(k;i)z_i^T ((A + BK)^T )^k } & 0  \\
\end{array}} \right]\\
& + \left[ {\begin{array}{*{20}c}
   0 & {{\cal O}_k u_k u(k;i)^T }  \\
   {u(k;i) u_k^T {\cal O}_k^T } & 0  \\
\end{array}} \right]
\end{align*}
where the first term on the right-hand side is bounded as
\begin{align}
&\left[ {\begin{array}{*{20}c}
   0 & {(A + BK)^k z_i u(k;i)^T }  \\
   {u(k;i)z_i^T ((A + BK)^T )^k } & 0  \\
\end{array}} \right]\nonumber\\
 =& \left[ {\begin{array}{*{20}c}
   {(A + BK)^k z_i }  \\
   0  \\
\end{array}} \right]\left[ {\begin{array}{*{20}c}
   0  \\
   {u(k;i)}  \\
\end{array}} \right]^T\nonumber\\
&+ \left[ {\begin{array}{*{20}c}
   0  \\
   {u(k;i)}  \\
\end{array}} \right]\left[ {\begin{array}{*{20}c}
   {(A + BK)^k z_i }  \\
   0  \\
\end{array}} \right]^T\nonumber\\
\succeq &  - \varepsilon \left[ {\begin{array}{*{20}c}
   0  \\
   {u(k;i)}  \\
\end{array}} \right]\left[ {\begin{array}{*{20}c}
   0  \\
   {u(k;i)}  \\
\end{array}} \right]^T\nonumber\\
& - \frac{1}{\varepsilon }\left[ {\begin{array}{*{20}c}
   {(A + BK)^k z_i }  \\
   0  \\
\end{array}} \right]\left[ {\begin{array}{*{20}c}
   {(A + BK)^k z_i }  \\
   0  \\
\end{array}} \right]^T\label{eq:11}\\
\succeq &  - \left[ {\begin{array}{*{20}c}
   {\varepsilon ^{ - 1} \lambda _{\max } ((A + BK)^k z_i z_i^T ((A + BK)^T )^k )}  \\
   0  \\
\end{array}} \right.\nonumber\\
&\left. {\begin{array}{*{20}c}
   0  \\
   {\varepsilon u(k;i)u(k;i)^T }  \\
\end{array}} \right]\nonumber\\
=&  - \left[ {\begin{array}{*{20}c}
   {\varepsilon ^{ - 1} \| (A + BK)^k z_i\|^2 I} & 0  \\
   0 & {\varepsilon u(k;i)u(k;i)^T }  \\
\end{array}} \right]\nonumber\\
\succeq&  - \left[ {\begin{array}{*{20}c}
   {\varepsilon \|(A + BK)^k\|^2 I} & 0  \\
   0 & {\varepsilon u(k;i)u(k;i)^T }  \\
\end{array}} \right]\nonumber
\end{align}
where~\eqref{eq:11} is due to~\cref{lemma:bounding-lemma} and the last inequality is due to $\|z_i\| \le \varepsilon$. Combining the two lower bounds, we have
\begin{align*}
&\left[ {\begin{array}{*{20}c}
   {x(k;i)x(k;i)^T } & {x(k;i)u(k;i)^T }  \\
   {u(k;i)x(k;i)^T } & {u(k;i)u(k;i)^T }  \\
\end{array}} \right]\\
\succeq & \left[ {\begin{array}{*{20}c}
   {(A + BK)^k z_i z_i^T ((A + BK)^T )^n  - 2\varepsilon I\|(A + BK)^k\|^2 } & 0  \\
   0 & 0  \\
\end{array}} \right]\\
&+ \underbrace {\left[ {\begin{array}{*{20}c}
   {(1 - \varepsilon) {\cal O}_k u_{k;i} u_{k;i}^T {\cal O}_k^T } & {{\cal O}_k u_{k;i} u(k;i)^T }  \\
   {u(k;i)u_{k;i}^T {\cal O}_k^T } & {(1 - \varepsilon )u(k;i)u(k;i)^T }  \\
\end{array}} \right]}_{ = :M_i }
\end{align*}

Therefore,
\begin{align*}
S_N =& \frac{1}{N}\sum_{i = 1}^N {\left( {\sum\limits_{k = 0}^n {\left[ {\begin{array}{*{20}c}
   {x(k;i)}  \\
   {u(k;i)}  \\
\end{array}} \right]\left[ {\begin{array}{*{20}c}
   {x(k;i)}  \\
   {u(k;i)}  \\
\end{array}} \right]^T } } \right)}\\
\succeq& \frac{1}{N}\sum_{i = 1}^N {\left[ {\begin{array}{*{20}c}
   {x(n;i)}  \\
   {u(n;i)}  \\
\end{array}} \right]\left[ {\begin{array}{*{20}c}
   {x(n;i)}  \\
   {u(n;i)}  \\
\end{array}} \right]^T }\\
\succeq& \left[ {\begin{array}{*{20}c}
   { - 2\varepsilon I \|(A + BK)^k\|^2 } & 0  \\
   0 & 0  \\
\end{array}} \right] + \frac{1}{N}\sum\limits_{i = 1}^N {M_i }
\end{align*}
Since $(M_1,M_2,\ldots,M_N)$ are i.i.d. random variables with mean
\[
{\mathbb E}[M_i ] = \left[ {\begin{array}{*{20}c}
   {(1 - \varepsilon ){\cal O}_n {\cal U}_n {\cal O}_n^T } & 0  \\
   0 & {(1 - \varepsilon )U}  \\
\end{array}} \right]
\]
where ${\cal U}_k$ is defined in~\cref{eq:13}. From the strong law of large numbers, with $\varepsilon \in (0,1)$, we have ${\mathbb P}\left[ \lim_{N \to \infty } \frac{1}{N}\sum_{i = 1}^N {M_i }  > 0 \right] = 1$. Therefore, for a sufficiently small $\varepsilon \in (0,1)$, $S_N$ converges to a positive definite matrix with probability one.
\end{proof}

\section{Conclusion}
We have developed data-driven control evaluation and design strategies based on LMIs and dynamic programming, where stabilization and LQR problems are addressed. Efficient data collection schemes have been investigated. Finally, we investigate exploration schemes to acquire valid data from the trajectories under different scenarios with theoretical guarantees of convergence. In particular, we prove that as more data is accumulated, the collected data becomes valid for the proposed algorithms with higher probability.

\bibliographystyle{IEEEtran}
\bibliography{reference}

\appendices


\end{document}